\documentclass[12pt,a4paper,reqno]{amsart}
\usepackage{amsmath, amsthm, amscd, amsfonts,graphicx}

\usepackage{color}
\usepackage{enumerate}
\numberwithin{equation}{section}
\theoremstyle{plain}
 \newtheorem{theorem}{Theorem}[section]
 \newtheorem{lemma}[theorem]{Lemma}

 \newtheorem{corollary}[theorem]{Corollary}
\theoremstyle{definition}
 
 \newtheorem{remark}[theorem]{Remark}
\theoremstyle{remark}
 \newtheorem{case}{Case}
 
\newenvironment{enumeratei}{\begin{enumerate}[\quad\upshape (i)]} {\end{enumerate}}
\DeclareMathOperator \Sub {Sub}
\DeclareMathOperator \Con {Con}
\newcommand \NS {NS}
\newcommand\nsub [1] {\textup{\NS}(#1)}
\newcommand\ncon [1] {\textup{NC}(#1)}
\newcommand\url [1] {\texttt{#1}}
\newcommand \ssty [1] {\scriptscriptstyle{#1}}
\newcommand \gluplus {\mathop{+_{\ssty{glu}}}}
\newcommand \chain [1] {C^{(#1)}}
\newcommand \ideal {\mathord{\downarrow}}
\newcommand \filter {\mathord{\uparrow}}
\newcommand \then {\Rightarrow}
\newcommand \achs [1] {$#1$-antichain}
\newcommand \achp [1] {$#1$-antichains}
\newcommand \fjsl [1] {F_{\textup{jsl}}(#1)}
\newcommand \fpl [1] {F_{\textup{lat}}(#1)}
\newcommand\set [1]{\{#1\}}
\newcommand \tuple [1] {\langle #1 \rangle}
\newcommand \pair [2] {\tuple{#1,#2}}
\newcommand \nplu {\mathbb N^+}
\newcommand \tbf[1] {\textbf{#1}}
\renewcommand \phi {\varphi}

\newcommand \badgood[2]{{#2}} 
\newcommand \red [1] {{\color{red}#1\color{black}}}

\newcommand \nothing [1] {}

\begin{document}
\title[A note on lattices with many sublattices]
{A note on lattices with many sublattices}


\author[G.\ Cz\'edli]{G\'abor Cz\'edli}
\address{University of Szeged, Bolyai Institute, Szeged,
Aradi v\'ertan\'uk tere 1, Hungary 6720}
\email{czedli@math.u-szeged.hu}
\urladdr{http://www.math.u-szeged.hu/\textasciitilde{}czedli/}

\author[E.\,K.\ Horv\'ath]{Eszter\ K.\ Horv\'ath}
\address{University of Szeged, Bolyai Institute\\Szeged, Aradi v\'ertan\'uk tere 1, Hungary 6720}
\email{horeszt@math.u-szeged.hu}
\urladdr{http://www.math.u-szeged.hu/\textasciitilde{}horvath/}

\thanks{This research of was supported by the Hungarian Research Grants KH 126581.}



\begin{abstract} For  every natural number $n\geq 5$, we prove that the number of
subuniverses of an $n$-element lattice is $2^n$, 
$13\cdot 2^{n-4}$, $23\cdot 2^{n-5}$, or less than $23\cdot 2^{n-5}$. By a subuniverse, we mean a sublattice or the emptyset.
Also, we describe the $n$-element lattices with exactly $2^n$, 
$13\cdot 2^{n-4}$, or $23\cdot 2^{n-5}$ \badgood{sublattices}{subuniverses}.
\end{abstract}

\date{\red{\hfill   December 30, 2018}}

\maketitle

\section{Introduction and our result}
For a lattice $L$, $\Sub(L)$ will denote its  \emph{sublattice lattice}. In spite of this standard terminology, $\Sub(L)$  consists of all \emph{subuniverses} of $L$. That is, a subset $X$ of $L$ is in $\Sub(L)$ iff $X$ is closed with respect to join and meet. In particular, $\emptyset\in\Sub(L)$. Note that for $X\in \Sub(L)$, $X$ is a sublattice of $L$ if and only if $X$ is nonempty. 
All lattices occurring in this paper will be assumed to be \emph{finite} even if this is not always emphasized.  
For a natural number $n\in\nplu:=\set{1,2,3,\dots}$, let 
\begin{equation*}
\nsub n:=\set{|\Sub(L)|: L\textup{ is a lattice of size }|L|=n}
.
\end{equation*}
That is, $k\in \nsub n$ if and only if some $n$-element lattice has exactly $k$ subuniverses. Although the acronym \NS{} comes from Number of Sublattices,  $L$ has only  $|\Sub(L)|-1$ \emph{sublattices}. 
If $K$ and $L$ are finite lattices, then their \emph{glued sum} $K\gluplus L$  is the ordinal sum of the posets $K\setminus{1_K}$, the singleton lattice, and $L\setminus\set{0_L}$, in this order. In other words, we put $L$ atop $K$ and identify the elements $1_K$ and $0_L$; see Figure~\ref{figone}. For example, if each of $K$ and $L$ is the two-element chain, then  $K\gluplus L$ is the three-element chain. Note that $\gluplus$ is an associative but not \badgood{}{a }commutative operation.

\begin{figure}[htb] 
\centerline
{\includegraphics[scale=1.1]{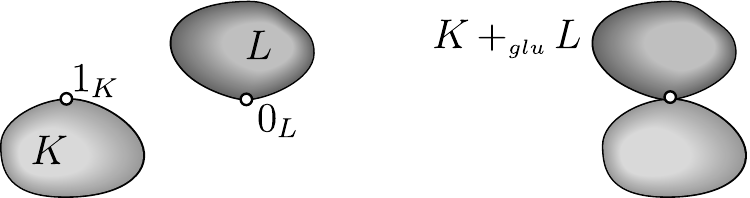}}
\caption{The glued sum $K\gluplus L$ of $K$ and $L$
\label{figone}}
\end{figure}%

Our goal is to prove the following result; the \emph{four-element boolean lattice} $B_4$ and the \emph{pentagon} (lattice) $N_5$ are given in Figure~\ref{figtwo}.

\begin{theorem}\label{thmmain}
If $5\leq n\in\nplu$, then the following three assertions hold.
\begin{enumeratei}
\item\label{thmmaina} The largest number in $\nsub n$ is $2^n=32\cdot 2^{n-5}$. Furthermore, an $n$-element lattice $L$ has exactly $2^n$ subuniverses if an only if $L$ is a chain.
\item\label{thmmainb}
 The second largest number in $\nsub n$ is $26\cdot 2^{n-5}$. Furthermore, an $n$-element lattice $L$ has exactly $26\cdot 2^{n-5}$ subuniverses if and only if $L\cong C_1\gluplus B_4\gluplus C_2$, where $C_1$ and $C_2$ are chains.
\item\label{thmmainc}  The third largest number in $\nsub n$ is $23\cdot 2^{n-5}$. Furthermore, an $n$-element lattice $L$ has exactly $23\cdot 2^{n-5}$ subuniverses if and only if $L\cong C_0\gluplus N_5\gluplus C_1$, where \badgood{$C_1$ and $C_2$}{$C_0$ and $C_1$} are chains.
\end{enumeratei}
\end{theorem}

\begin{figure}[htb] 
\centerline
{\includegraphics[scale=1.1]{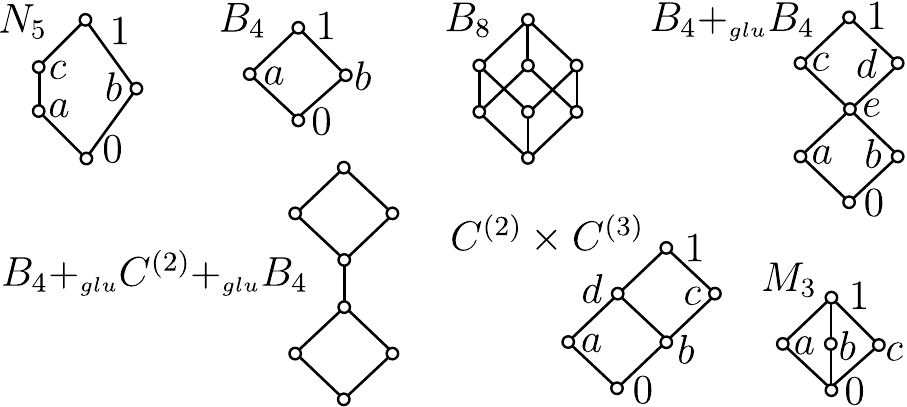}}
\caption{Lattices for Theorem~\ref{thmmain} and Lemma~\ref{lemmaLHFtfsZh}
\label{figtwo}}
\end{figure}%

Since $\nsub n=\set{2^n}$ for $n\in\set{1,2,3}$ and
$\nsub 4 =\set{13,16}$, we have formulated this theorem only for $n\geq 5$. To make the comparison of the numbers occurring in the paper easier, we often give $|\Sub(L)|$ as a multiple of $2^{|L|-5}$. 
Next, we repeat the first sentence of the Abstract, which  is a trivial consequence of Theorem~\ref{thmmain}.

\begin{corollary}
For $5\leq n\in\nplu$,  \badgood{we prove that}{} 
the number of
subuniverses of an $n$-element lattice is $2^n$, 
$13\cdot 2^{n-4}$, $23\cdot 2^{n-5}$, or less than $23\cdot 2^{n-5}$.
\end{corollary}

For $k\in\nplu$, the $k$-element chain will be denoted by $\chain k$.

\begin{remark}\label{remarkclJl}
Let $\Con(L)$ and $\ncon n$ stand for the lattice of congruences of a lattice $L$ and $\set{|\Con(L)|: L\text{ is a lattice with }|L|=n}$, respectively. For $n\geq 5$, the five largest numbers in $\ncon n$
are $16\cdot 2^{n-5}$, $8\cdot 2^{n-5}$, $5\cdot 2^{n-5}$,
$4\cdot 2^{n-5}$, and $3.5\cdot 2^{n-5}$ by Freese~\cite{freesecomplat}, Cz\'edli~\cite{czglatmanycongr} and, mainly, Kulin and Mure\c san~\cite{kulinmuresan}. 
\end{remark}

\begin{remark}\label{remarkfrthl}
Interestingly, the first three of the five numbers mentioned in Remark~\ref{remarkclJl} are witnessed exactly by the same lattices that occur in Theorem~\ref{thmmain}. However, we will show at the end of Section~\ref{sectionprepare} that
\begin{align}
|\Sub(N_5\gluplus \chain 3)|=23\cdot 2^{7-5} > 21.25\cdot 2^{7-5}\cr
=|\Sub(B_4\gluplus B_4)|>19\cdot 2^{7-5}\cr
=|\Sub((\chain2\times\chain3)\gluplus\chain 2)|,
\label{alignZhfTnhGRb}
\end{align}
which indicates that $|\Sub((\chain2\times\chain3)\gluplus\chain 2)|$ is not the fourth largest number in $\nsub 7$, although we know from 
Kulin and Mure\c san~\cite{kulinmuresan} that $|\Con((\chain2\times\chain3)\gluplus\chain 2)|$ is the fourth largest number in $\ncon 7$.
\end{remark}

While there are powerful tools to determine the first few large numbers in $\ncon n$, see the above-mentioned papers and, for additional tools,  Cz\'edli~\cite{czglatmancplanar}, the analogous task for $\nsub n$ seems to be more tedious.  This together with Remark~\ref{remarkfrthl} are our excuses that we do not \badgood{determined}{determine} the fourth and fifth largest numbers in $\nsub n$.

The rest of the paper is devoted to the proof of Theorem~\ref{thmmain}.

\section{Two preparatory lemmas}\label{sectionprepare}
Our notation and terminology is standard, see, for example, 
Gr\"atzer \cite{ggglt}, or see its freely available part at \url{tinyurl.com/lattices101}. However, we recall some notation and introduce some auxiliary concepts. 
For elements $u,v$ in a lattice $L$, the \emph{interval} $[u,v]:=\set{x\in L: u\leq x\leq v}$ is defined only if $u\leq v$, but the \emph{sublattice} $[\set{u,v}]$
generated by $\set{u,v}$ always makes sense. In order to avoid confusion,
the curly brackets are never omitted from $[\set{a_1,\dots,a_k}]$ when a generated sublattice is mentioned. 
For $u\in L$,
the principal ideal and the principal filter generated by $u$ are $\ideal u:=\set{x\in L: x\leq u}$ and  $\filter u:=\set{x\in L: u\leq x}$, respectively. We can also write $\ideal_L u$ and $\filter_L v$ to specify the lattice $L$.
For $u,v\in L$, we write $u\parallel v$ if $u$ and $v$ are \emph{incomparable}, that is, $u\not\leq v$ and $v\not\leq u$.  We say that $u$ is \emph{join-irreducible} if $u$ has at most one lower cover; note that \badgood{$1=1_L$}{$0=0_L$} is  join-irreducible by our convention. \emph{Meet-irreducibility} is defined dually, and an element is \emph{doubly irreducible} if it is both join-irreducible and meet-irreducible.  Next, let us call an element $u\in L$  \emph{isolated} if $u$ is doubly irreducible and $L=\ideal u\cup\filter u$. That is, if $u$ is  \badgood{double}{doubly} irreducible and  $x\parallel u$ holds for no $x\in L$. Finally, an interval $[u,v]$ will be called an \emph{isolated edge} if it is a prime interval, that is, $u\prec v$, and $L=\ideal u\cup \filter v$.

\begin{lemma}\label{lemmasublat}  If $K$ is a sublattice and $H$ is a subset of a finite lattice $L$, then the following three assertions hold.
\begin{enumeratei}
\item\label{lemmasublata} With the notation $t:=|\set{H\cap S: S\in \Sub(L)}|$, we have that $|\Sub(L)|\leq t\cdot 2^{|L|-|H|}$. 
\item\label{lemmasublatb} $|\Sub(L)|\leq |\Sub(K)|\cdot 2^{|L|-|K|}$.
\item\label{lemmasublatc} Assume, in addition, that $K$ has neither an isolated element, nor an isolated edge. Then  $|\Sub(L)| = |\Sub(K)|\cdot 2^{|L|-|K|}$ if and only if  $L$ is (isomorphic to) \badgood{$C_0\gluplus K\gluplus C_1$ for some chains $C_1$ and $C_2$}{$C_0\gluplus K\gluplus C_1$ for some chains $C_0$ and $C_1$}.
\end{enumeratei}
\end{lemma}

\begin{proof}
With respect to the map $\phi\colon \Sub(L)\to \set{H\cap S: S\in \Sub(L)}$, defined by $X\mapsto H\cap X$, each $Y\in \set{H\cap S: S\in \Sub(L)}$ has at most $2^{|L|-|H|}$ preimages. This yields part \eqref{lemmasublata}.
Clearly, \eqref{lemmasublata} implies  \eqref{lemmasublatb}.
The argument above yields a bit more than stated in \eqref{lemmasublata}  and \eqref{lemmasublatb}; namely, for later reference, note the following.
\begin{equation}
\parbox{7.8cm}{If $|\Sub(L)| = |\Sub(K)|\cdot 2^{|L|-|K|}$, then for every $S\in\Sub(K)$ and every subset $X$ of $L\setminus K$, we have that $S\cup X\in\Sub(L)$.}
\label{pbxZuGrthGCw}
\end{equation}


Next, we claim that for an element $u\in L$, 
\begin{equation}
\parbox{9.3cm}{$u$ is isolated if and only if for every $X\in\Sub(L)$, we have that $X\cup\set u\in\Sub(L)$ and  $X\setminus\set u\in\Sub(L)$.}
\label{pbxUsLtDhmB}
\end{equation}
Assume that $u$ is isolated and $X\in\Sub(L)$. Since $u$ is doubly irreducible, $X\setminus\set u\in\Sub(L)$.  Since $u$ is comparable with all elements of $X$, $X\cup\set u\in\Sub(L)$, proving the  ``only if'' part of \eqref{pbxUsLtDhmB}. To show the converse, assume that $u$ is not isolated. If $u$ is not doubly irreducible, then $u=a\vee b$ with $a,b<u$ or dually, 
and $X:=\set{a,b,u,a\wedge b}\in \Sub(L)$ but $X\setminus\set u\notin\Sub(L)$. If $u\parallel v$ for some $v\in L$, then 
$\set v\in\Sub(L)$ but $\set v \badgood{\cap}{\cup} \set u\notin \Sub(L)$. This proves the ``if'' part, and  \eqref{pbxUsLtDhmB} has been verified.

Next, to prove part \eqref{lemmasublatc}, assume that $K$ has neither an isolated element, nor an isolated edge. First, let $L=C_1\gluplus K\gluplus C_2$. Since every $u$ in $L\setminus K$ is clearly an isolated element of $L$, it follows from a repeated application of  \eqref{pbxUsLtDhmB} that whenever $X\subseteq L\setminus K$ and $S\in\Sub(K)$, then $S\cup X\in \Sub(L)$. Since $L\setminus K$ has  $2^{|L|-|K|}$ subsets, $|\Sub(L)|\geq |\Sub(K)|\cdot 2^{|L|-|K|}$, and we obtain the required equality by the converse inequality given in part \eqref{lemmasublatb}.

Conversely, assume the equality given in \eqref{lemmasublatc}. Let $x$ be an arbitrary element of $L\setminus K$.  Applying \eqref{pbxZuGrthGCw} to $\set{0_K}\in\Sub(K)$ and $\set{1_K}\in\Sub(K)$, we obtain that  both $\set{0_K,x}$ and $\set{1_K,x}$ are in $\Sub(L)$, whence neither $x\parallel 0_K$, nor $x\parallel 1_K$.  So exactly one of the cases $0_K<x<1_K$, $x<0_K$, and $1_K<x$ holds; we are going to exclude the first one. Suppose for a contradiction that $0_K<x<1_K$. Then $x$ is comparable to every $y\in K$, because otherwise $S:=\set{y}$ and $\set x$ would violate \eqref{pbxZuGrthGCw}. By finiteness, we can take $u:=\bigvee (K\cap \ideal x)$ and  $v:=\bigwedge (K\cap \filter x)$. Now if $y\in K$, then either $y>x$ and so $y\in \filter_K v$, or $y<x$ and so $y\in \ideal_K u$, which means that $K=\ideal_K u \cup \filter_K v$. Hence, $[u,v]_K$ is  a prime interval of $K$, and so it is an isolated edge of $K$. This  is a contradiction, which  excludes that $0_K<x<1_K$. Therefore, with the notation $C_0:=\ideal_L 0_K$ and  $C_1:=\filter_L 1_K$, we obtain that $L$ is (isomorphic to)
$C_0\gluplus K\gluplus C_1$. Consequently, in order to show that $C_0$ and $C_1$ are chains and to complete the proof, it suffices to show that every $u\in L\setminus K$ is an isolated element of $L$. Suppose the contrary. Then \eqref{pbxUsLtDhmB} yields a subuniverse $Y\in\Sub(L)$ such that 
\begin{equation}
Y\cup\set u\notin \Sub(L)\quad\text{ or }\quad Y\setminus\set u\notin \Sub(L).
\label{eqgkbnnSTkLkk}
\end{equation}
Since \badgood{$u\notin L$}{$u\notin K$}, 
we have that $Y\cap K=(Y\cup\set u)\cap K = (Y\setminus\set u)\cap K$; we denote this set by $S$. Then $S\in \Sub(K)$ since $Y\in\Sub(L)$. It follows from \eqref{pbxZuGrthGCw} that
\begin{align*}
Y\cup \set u&= S\cup ((Y\cup \set u)\setminus K)\in \Sub(L)
\text{ and }\cr
 Y\setminus \set u&= S\cup ((Y\setminus \set u)\setminus K)\in \Sub(L),
\end{align*}
which contradicts \eqref{eqgkbnnSTkLkk} and completes the proof of Lemma~\ref{lemmasublat}
\end{proof}


The following lemma is easier and even a computer program could prove it. For the reader's convenience, we give its short proof. The \badgood{minuend}{minuends} in the exponents will be the sizes of the lattices in question.

\begin{lemma}\label{lemmaLHFtfsZh}

 For the lattices given in Figure~\ref{figtwo}, the following \badgood{six}{seven} assertions hold.
\begin{enumeratei}
\item\label{lemmaLHFtfsZha} $|\Sub(B_4)|=13=26\cdot 2^{4-5}$.
\item\label{lemmaLHFtfsZhb}  $|\Sub(N_5)|=23=23\cdot 2^{5-5}$.
\item\label{lemmaLHFtfsZhc}  $|\Sub(\chain2\times \chain3)|=38=19\cdot 2^{6-5}$.
\item\label{lemmaLHFtfsZhd}  $|\Sub(B_4\gluplus B_4)|=85=21.25\cdot 2^{7-5}$.
\item\label{lemmaLHFtfsZhe}  $|\badgood{\Sub(\Sub(B_4\gluplus \chain2\gluplus B_4))}{\Sub(B_4\gluplus \chain2\gluplus B_4)}|=169=21.125\cdot 2^{8-5}$.
\item\label{lemmaLHFtfsZhf}  $|\Sub(M_3)|=20=20\cdot 2^{5-5}$.
\item\label{lemmaLHFtfsZhg}  $|\Sub(B_8)|=74=9.25\cdot 2^{8-5}$.
\end{enumeratei}
\end{lemma}

\begin{proof} The notation given by Figure~\ref{figtwo} will extensively be used.

Among all subsets of $B_4$, only $\set{a,b}$, $\set{a,b,0}$, and $\set{a,b,1}$ are \emph{not} subuniverses; this proves \eqref{lemmaLHFtfsZha}. For later reference, note that
\begin{equation}
\text{if $L$ is a chain, then $|\Sub(L)|=2^{|L|}$.}
\label{eqtxtFChncn}
\end{equation}
Implicitly, Lemma~\ref{lemmasublat}\eqref{lemmasublatc} will often be  used below.
%
Observe that
\begin{align*}
&|\set{S\in \Sub(N_5): \set{a,c}\cap S=\emptyset}|=\badgood{9}8,\qquad\text{by \eqref{eqtxtFChncn},}\cr
&|\set{S\in \Sub(N_5): \set{a,c}\cap S\neq\emptyset},\,\, b\notin S|=3\cdot 4=12, \text{ and}\cr
&|\set{S\in \Sub(N_5): \set{a,c}\cap S\neq\emptyset},\,\, b\in S|=3,
\end{align*}
whereby $|\Sub(N_5)|=\badgood{9}8+12+3=\badgood{24}{23}$ proves \eqref{lemmaLHFtfsZhb}.
Next, $S$ will belong to $\Sub(\chain 2\times \chain 3)$ even if this is not indicated. Let us compute:
\begin{align*}
&|\set{S: a\notin S}|=26,\qquad\text{by Lemmas \ref{lemmasublat}\eqref{lemmasublatc} and \ref{lemmaLHFtfsZh}\eqref{lemmaLHFtfsZha},}\cr
&|\set{S: a, b\in S|}=3, \text{ since }0,d\in S \text{ and }c\in S\then 1\in S,\cr
&|\set{S: a\in S,\,\,b\notin S,\,\,c\in S|}=1, \text{ since }0,1\in S \text{ and }\badgood{x}d\notin S,\cr
&|\set{S: a\in S,\,\,b\notin S,\,\,c\notin S|}=8, \text{by \eqref{eqtxtFChncn}}.
\end{align*}
Hence, $26+3+1+8=\badgood{28}{38}$ proves \eqref{lemmaLHFtfsZhc}.
Next, $S$ will automatically belong to $\Sub(B_4\gluplus B_4)$. We have that $|\set{S: \set{a, b}\subseteq S|}=7$,
because then $\set{c,d}\subseteq S\then 1\in S$ and $0,e\in S$. Also, $|\set{S: \set{a, b}\not\subseteq S|}=13\cdot 3\cdot 2=78$, because Lemma~\ref{lemmaLHFtfsZh}\eqref{lemmaLHFtfsZha} applies to the upper $B_4$, there are 3 possibilities for $a$ and $b$, and two for 0. Hence, 78+7=85 proves \badgood{\eqref{lemmaLHFtfsZhc}}{\eqref{lemmaLHFtfsZhd}}.
For $S\in \Sub(B_4\gluplus \chain2\gluplus B_4)$, the intersection of $S$ with the lower $B_4$ and that with the upper $B_4$ can  independently\badgood{}{be}{} chosen. Therefore, \eqref{lemmaLHFtfsZhe} follows from \eqref{lemmaLHFtfsZha}.

Next, we count the subuniverses $S$ of $M_3$. There are 4 with the property $|\set{a,b,c}\cap \badgood{M_3}S|\geq 2$, because they contain 0 and 1. There are $3\cdot 4=12$ with $|\set{a,b,c}\cap \badgood{M_3}S|=1$, and 4 with $|\set{a,b,c}\cap \badgood{M_3}S|=0$. Thus, $|\Sub(M_3)|=4+12+4=20$, proving  \eqref{lemmaLHFtfsZhf}.

The argument for $B_8$ is more tedious. It has 9 at most one-element subuniverses. There are 12 edges. We have 6 two-element subuniverses in which the heights of the two elements differ by two and 1 in which this difference is three. We have 12 three-element covering chains and 6 non-covering ones. The number of four-element (necessarily covering) chains is $3\cdot 2=6$, $B_4$ is embedded in 6 cover-preserving ways and (thinking of pairs of complementary elements) in 3 additional ways. 
The five-element sublattices are obtained \badgood{form}{from} cover-preserving $B_4$-sublattices by adding the unique one of $0_{B_8}$ or $1_{B_8}$ that is missing; their number is 6. To each of the $B_4$-sublattices at the bottom we can glue  a $B_4$-sublattice at the top in two ways, whence there are exactly 6  six-element subuniverses. In absence of doubly irreducible elements, there is no seven-element sublattice, and there is 1 eight-element one. The \badgood{some}{sum} of the numbers we have listed is \badgood{76}{74}, proving \eqref{lemmaLHFtfsZhg} and Lemma~\ref{lemmaLHFtfsZh}
\end{proof}

Now, we are in the position to prove \eqref{alignZhfTnhGRb}, mentioned in Remark~\ref{remarkfrthl}.

\begin{proof}[Proof of Remark~\ref{remarkfrthl}]
Combine Lemma~\ref{lemmasublat}\eqref{lemmasublatc} with parts
\eqref{lemmaLHFtfsZhb}, \eqref{lemmaLHFtfsZhc}, and \badgood{\eqref{lemmaLHFtfsZhe}}{\eqref{lemmaLHFtfsZhd}} of  Lemma~\ref{lemmaLHFtfsZh}.
\end{proof}

\section{The rest of the proof}
For brevity, a $k$-element antichain will be called a \emph{\achs{k}}. First, we recall two well-known facts from the folklore.

\begin{lemma}\label{lemmafreeslat}
For every join-semilattice $\badgood L S$ generated by $\set{a,b,c}$, there is
a unique surjective  homomorphism $\phi$ from the free join-semilattice $\fjsl{\tilde a,\tilde b, \tilde c}$, given in Figure~\ref{figthree}, onto $S$ such that $\phi(\tilde a)=a$, $\phi(\tilde b)=b$, and $\phi(\tilde c)=c$. 
\end{lemma}

\begin{figure}[htb] 
\centerline
{\includegraphics[scale=1.1]{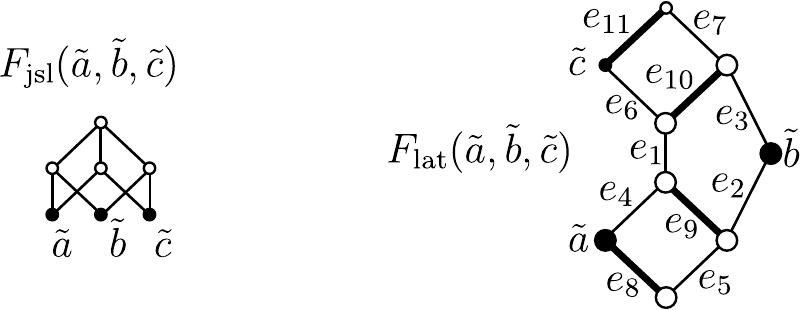}}
\caption{$\fjsl{\tilde a,\tilde b, \tilde c}$ and $\fpl{\tilde a,\tilde b, \tilde c}$
\label{figthree}}
\end{figure}%

\begin{lemma}[{Rival and Wille~\cite[Figure 2]{rivalwille}}]\label{lemmafreeLt}
For every lattice  $K$ generated by $\set{a,b,c}$ such that $a<c$, there is a unique surjective  homomorphism $\phi$ from the finitely presented lattice  $\fpl{\tilde a,\tilde b, \tilde c}$, given in Figure~\ref{figthree},  onto $\badgood{L}K$ such that $\phi(\tilde a)=a$, $\phi(\tilde b)=b$, and $\phi(\tilde c)=c$. 
\end{lemma}

We are going to use the two lemmas above in the proof of the following lemma. Implicitly, we will often use the well-known Homomorphism Theorem; see, e.g., Burris and Sankappanavar~\cite[Theorem 6.12]{burrissankappanavar}.

\begin{lemma}\label{lemmaXhrmlNc}
If an $n$-element lattice $L$ has a \achs3, then we have that 
$|\Sub(L)|\leq 20\cdot 2^{n-5}$.
\end{lemma}

\begin{proof}
Let $\set{a,b,c}$ be a \achs3 in $L$. 
Lemma~\ref{lemmafreeslat} yields a unique  join-homomorphism 
from $\fjsl{(\tilde a,\tilde b,\tilde c})$ to $S:=\set{a,b,c,a\vee b, a\vee c, b\vee c, a\vee b\vee c}$ such that $\phi$ maps to $\tilde a$, $\tilde b$, and $\tilde c$ to $a$, $b$, and $c$, respectively. 
Since $\set{a,b,c}$ is an antichain, none of the \badgood{eight}{six} lower edges of $\fjsl{(\tilde a,\tilde b,\tilde c})$ is collapsed by the kernel $\Theta:=\ker(\phi)$ of $\phi$. Hence,  there are only four cases for the join-subsemilattice $S\cong \fjsl{(\tilde a,\tilde b,\tilde c})/\Theta$ of $L$, depending on the number the upper edges collapsed by $\Theta$.

\begin{case}[none of the three upper edges is collapsed by $\Theta$]\label{caseone} Then \badgood{$s$}{$S$} is isomorphic to $\fjsl{(\tilde a,\tilde b,\tilde c})$, whereby $\set{a\vee b, a\vee c,b\vee c}$ is a \achs3. We know from, say, Gr\"atzer~\cite[Lemma 73]{ggglt}, that this \achs3 generates a sublattice isomorphic to $B_8$. Hence, $|\Sub(L)|\leq 9.25\cdot 2^{n-5}\leq 20\cdot 2^{n-5}$ by Lemmas~\ref{lemmasublat}\eqref{lemmasublatb} and \ref{lemmaLHFtfsZh}\eqref{lemmaLHFtfsZhg}, as required.
\end{case}

\begin{figure}[htb] 
\centerline
{\includegraphics[scale=1.1]{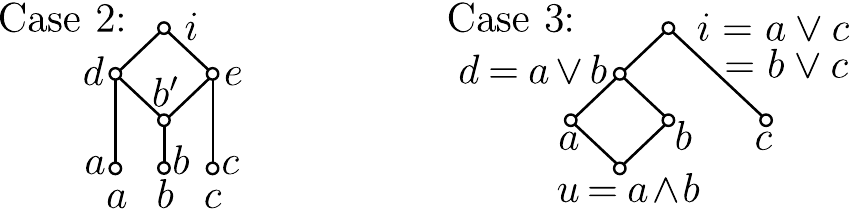}}
\caption{Cases~\ref{caseket} and \ref{casehar} \label{figfour}}
\end{figure}%

\begin{case}[$\Theta$ collapses exactly one upper edge]\label{caseket}
Apart from notation, we have that $d:=a\vee b <a\vee c=:i$ and  $e:=b\vee c < i$; see Figure~\ref{figfour} on the left. Letting $b':=d\wedge e$, we have that $a\vee b'=d$ and $b'\vee c=e$. Since $b\leq b'$ and $b\not\leq a$, we have that $b'\not\leq a$. If we had $a\leq b'$, then 
$i=d\vee e=a\vee b'\vee e=b' \vee e =e$
would be a contradiction. Hence, $a\parallel b'$, and $\set{a,b',c}$ is a \achs3 by $a$--$c$ symmetry.
We can count the subuniverses $T$ of the join-semilattice $H:=\set{a,b',c,d,e,i}$ as follows. We have that
$|\set{T: b'\notin T}|\leq 3\cdot7=21$, because $\set{d,e}\not\subseteq T$ allows only three possibilities for $T\cap \set{d,e}$ and $a\vee c=i$ at most seven possibilities for $T\cap \set{a,c,i}$. Similarly,
\begin{align*}
&|\set{T: b'\in T, a\notin T, c\notin T}|\leq 7,\quad\text{because }
\set{d,e}\subseteq T\then i\in T,\cr
&|\set{T: b'\in T, a\in T, c\notin T}|\leq 3,\quad\text{since }d\in T, \text{ so }e\in T\then i\in T,\cr
&|\set{T: b'\in T, a\notin T, c\in T}|\leq 3,\quad\text{by $a$--$c$ symmetry, and }\cr
&|\set{T: b'\in T, a\in T, c\in T}| =1,\quad\text{because }T=H.
\end{align*}
Note that some of the inequalities above are equalities, but we do not need this fact. 
Forming the sum of the above numbers, the join-semilattice $H$ has at most $35=17.5\cdot 2^{6-5}$ subuniverses. Hence, Lemma~\ref{lemmasublat}\eqref{lemmasublata} yields that  $\Sub(L)\leq 17.5\cdot 2^{n-5}\leq 20\cdot 2^{n-5}$, as required.
\end{case}

\begin{case}[$\Theta$ collapses two of the upper edges]\label{casehar}
Apart from notation, we have that $d:=a\vee b < a\vee c=b\vee c=:i$. Let $u:=a\wedge b$; see Figure~\ref{figfour}. We focus on possible intersections of subuniverses of $L$ with $H:=\set{a,b,c,d,u,i}$. Denoting such an intersection by $S$, we can compute as follows.
\begin{align*}
&|\set{S: c\notin S}|\leq 26, \quad \text{by Lemmas~\ref{lemmasublat}\eqref{lemmasublatb}  and  \ref{lemmaLHFtfsZh}\eqref{lemmaLHFtfsZha},}\cr
&|\set{S: c\in S, a\in S, b\in S}|\leq 1\quad \text{since }H\subseteq[\set{a,b,c}], \cr
&|\set{S: c\in S, a\in S, b\notin S}|\leq 4 \quad \text{since }i\in[\set{a,c}],\cr
&|\set{S: c\in S, a\notin S, b\in S}|\leq 4, \quad \text{by $a$--$b$-symmetry,}\cr
&|\set{S: c\in S, a\notin S, b\notin S, d\in S}|\leq 2, \quad \text{because }i\in S,\cr
&|\set{S: c\in S, a\notin S, b\notin S, d\notin S}|\leq 3 \quad \text{since }u\in S\then i\in S.
\end{align*}
Since the sum of these numbers is $40$, we obtain from Lemma~\ref{lemmasublat}\eqref{lemmasublata}  that $|\Sub(L)|\leq  40\cdot 2^{n-6}=  20\cdot 2^{n-5}$, as required.
\end{case}

\begin{case}[all the three upper edges are collapsed]\label{casenegy}
Clearly, $a\vee b=a\vee c=b\vee c=a\vee b\vee c=:i$. If $a\wedge b=a\wedge c=b\wedge c=a\wedge b\wedge c$ failed, then  the dual of one of the previous three cases would apply. Hence, we can assume that the sublattice $[\set{a,b,c}]$ generated by $\set{a,b,c}$ is isomorphic to $M_3$; see Figure~\ref{figtwo}. Therefore,
$|\Sub(L)|\leq 20\cdot 2^{n-5}$ by Lemmas~\ref{lemmasublat}\eqref{lemmasublatb} and \ref{lemmaLHFtfsZh}\eqref{lemmaLHFtfsZhf}, completing the proof of Case~\ref{casenegy} and that of Lemma~\ref{lemmaXhrmlNc}. \qedhere
\end{case}
\end{proof}

\begin{proof}[Proof of Theorem~\ref{thmmain}]
Part \eqref{thmmaina} is trivial. From Lemmas~\ref{lemmasublat} and \ref{lemmaLHFtfsZh}\eqref{lemmaLHFtfsZha}, we conclude part \eqref{thmmainb}. 
So, we are left only with part  \eqref{thmmainc}.

In what follows, let $L$ be an $n$-element lattice.
We obtain from Lemmas~\ref{lemmasublat}\eqref{lemmasublatc} and \ref{lemmaLHFtfsZh}\eqref{lemmaLHFtfsZhb} that if 
\begin{equation}
\text{$L\cong C_0\gluplus N_5\gluplus C_1$ for finite chains $C_0$ and $C_1$,}
\label{eqandtxtffChnS}
\end{equation}
then $|\Sub(L)|=\badgood{26}{23}\cdot 2^{n-5}$.  In order the complete the proof of Theorem~\ref{thmmain}, it suffices to exclude the existence of a lattice $L$ such that
\begin{equation}
\parbox{8cm}{$|L|=n$, $23\cdot 2^{n-5}\leq |\Sub(L)|<26\cdot 2^{n-5}$, but $L$ is not of the form given in \eqref{eqandtxtffChnS}.}
\label{pbxndrWdztN}
\end{equation}
Suppose, for a contradiction, that $L$ is a lattice satisfying \eqref{pbxndrWdztN}. Then, by  Theorem~\ref{thmmain} \eqref{thmmaina} and \eqref{thmmainb} and Lemma~\ref{lemmaXhrmlNc},
\begin{equation}
\text{$L$ has at least two \achp2 but it has no \achs3.}
\label{eqtxtghxzBmPwPms}
\end{equation}
We claim that
\begin{equation}
\text{$L$ cannot have two \emph{non-disjoint} \achp2.}
\label{eqtxtnmdszjlThlN}
\end{equation}
Suppose to the contrary that $\set{a,b}$ and $\set{c,b}$ are two distinct \achp2 in $L$. Since there is no \achs3 in $L$, we can assume that $a<c$. With $K:=[\set{a,b,c}]$, let $\phi\colon \fpl{\tilde a,\tilde b,\tilde c}\to K$ be the unique lattice homomorphism from Lemma~\ref{lemmafreeLt}, and let $\Theta$ be the kernel of $\phi$. 
We claim that $\Theta$ collapses $e_1$; see Figure~\ref{figthree}. Suppose to the contrary that $\Theta$ does not collapse $e_1$. Since $e_1$ generates the monolith congruence, that is, the smallest nontrivial congruence of the $N_5$ sublattice of $\fpl{\tilde a,\tilde b,\tilde c}$, no other \badgood{edges}{edge} of this $N_5$ sublattice is collapsed. Hence,
$N_5$ is a sublattice of $L$, and it follows from Lemmas~\ref{lemmasublat}\eqref{lemmasublatb} and \ref{lemmaLHFtfsZh}\eqref{lemmaLHFtfsZhb} that $|\Sub(L)|\leq 23\cdot 2^{n-5}$. Thus, \eqref{pbxndrWdztN} yields that $|\Sub(L)| = 23\cdot 2^{n-5}$.
Applying  \badgood{Lemmas}{Lemma}~\ref{lemmasublat}\eqref{lemmasublatc} for $K:=N_5$ and $L$, we obtain that $L$ is of the form \eqref{eqandtxtffChnS}. This contradicts \eqref{pbxndrWdztN}, and we have shown that $\Theta$ does collapse $e_1$. On the other hand, since $a\parallel b$ and $c\parallel b$, none of the thick edges $e_8,\dots,e_{11}$ is collapsed by $\Theta$. Observe that at least one of $e_4$ and $e_6$ is \emph{not} collapsed by $\Theta$, since otherwise $\pair{\tilde a}{\tilde c}$ would belong to $\Theta=\ker(\phi)$ by transitivity and $a=c$ would be a contradiction. 
By duality, we can assume that $e_4$ is not collapsed by $\Theta$.
Since $e_2$, $e_3$, and $e_5$ are perspective to $e_{10}$, $e_9$, and $e_4$, respectively, these edges are not collapsed either. So, with the exception of $e_1$, no edge among the ``big'' elements in Figure~\ref{figthree} is collapsed. Thus, the $\phi$-images of the elements denoted by big circles form a sublattice (isomorphic to) $\chain2\times \chain3$ in $L$. Hence, $|\Sub(L)|\leq 19\cdot 2^{n-5}$  by Lemmas~\ref{lemmasublat}\eqref{lemmasublatb} and \ref{lemmaLHFtfsZh}\eqref{lemmaLHFtfsZhc}, which contradicts our assumption that $L$ satisfies \eqref{pbxndrWdztN}. This proves \eqref{eqtxtnmdszjlThlN}.

To provide a convenient tool to exploit \eqref{eqtxtghxzBmPwPms} and \eqref{eqtxtnmdszjlThlN}, we claim that 
\begin{equation}
\parbox{9.0cm}{if $x,y,z\in L$ such that $|\set{x,y,z}|=3$ and $x\parallel y$, then either $\set{x,y}\subseteq\ideal z$, or $\set{x,y}\subseteq\filter z$.}
\label{eqpbxZhxszLn}
\end{equation}
To see this, assume the premise. Since $L$ has no \achs3,  $z$ is comparable to one of $x$ and $y$. By duality \badgood{an}{and} symmetry, we can assume that $x<z$. Since $z<y$ would imply $x<y$ and
$z\parallel y$ together with $x\parallel y$ would contradict \eqref{eqtxtnmdszjlThlN}, we have that $y<z$. This proves \eqref{eqpbxZhxszLn}.

Next, by  \eqref{eqtxtghxzBmPwPms} and \eqref{eqtxtnmdszjlThlN}, 
 we have a four-element subset $\set{a,b,c,d}$ of $L$ such that $a\parallel b$ and $c\parallel d$. By duality and \eqref{eqpbxZhxszLn}, 
we can assume that $\set{a,b}\subseteq\ideal c$. Applying \eqref{eqpbxZhxszLn} also to $\set{a,b,d}$, we obtain that $\set{a,b}$ is included either in $\filter d$, or in $\ideal d$. Since the first alternative would lead to
$d<a<c$ and so it would contradict $c\parallel d$, we have that $\set{a,b}\subseteq \badgood{\ideal c}{\ideal d}$. 
Thus, $\set{a,b}\subseteq \ideal c\cap \ideal d=\ideal(c\wedge d)$, 
and we obtain that $u:=a\vee b \leq c\wedge d=:v$. Let $S:=\set{a\wedge b,a,b,u,v, c,d, c\vee d}$. Depending on $u=v$ or $u<v$,
$S$ is a sublattice isomorphic to $B_4\gluplus B_4$ or 
$B_4\gluplus\chain2\gluplus B_4$. Using Lemma~\ref{pbxZuGrthGCw} together with \eqref{lemmaLHFtfsZhd} and \eqref{lemmaLHFtfsZhe} of 
Lemma~\ref{lemmaLHFtfsZh}, we  obtain that $|\Sub(L)|\leq 21.25\cdot 2^{n-5}$. This inequality contradicts \eqref{pbxndrWdztN} and completes the proof of Theorem~\ref{thmmain}.
\end{proof}


\begin{thebibliography}{99}

\bibitem{burrissankappanavar}
 Burris, S.; Sankappanavar, H. P.:
 A Course in Universal Algebra.
 A Course in Universal Algebra, Graduate Texts
in Mathematics, 78, Springer, New York–Berlin, 1981; The Millennium Edition,
\url{http://www.math.uwaterloo.ca/\textasciitilde{}snburris/htdocs/ualg.html}


\bibitem{czglatmanycongr}
  Cz\'edli, G.:
  A note on finite lattices with many congruences. 
  Acta Universitatis Matthiae Belii, Series
Mathematics, Online, 22--28,   
 \url{ http://actamath.savbb.sk/pdf/oacta2018003.pdf} 
(2018)

\bibitem{czglatmancplanar}
  Cz\'edli, G.:
   Lattices with many congruences are planar.
   \url{http://arxiv.org/abs/1807.08384}


\bibitem{freesecomplat}
  Freese, R.:
  Computing congruence lattices of finite lattices. 
  Proc. Amer. Math. Soc. \tbf{125},  3457--3463 (1997)


\bibitem{ggglt}
   Gr\"atzer, G.: 
  Lattice Theory: Foundation. Birkh\"auser Verlag, Basel (2011)

\bibitem{kulinmuresan}
  Kulin, J., Mure\c san, C.:
  Some extremal values of the number of congruences of a finite lattice,
  \url{https://arxiv.org/pdf/1801.05282} (2018)


\bibitem{muresan}
  Mure\c san, C.:
  Cancelling congruences of lattices while keeping their filters and ideals. 
  \url{https://arxiv.org/pdf/1710.10183} (2017)

\bibitem{rivalwille}
 Rival, I., Wille, R.:
 Lattices freely generated by partially ordered sets: which can be ``drawn''?.
  J. Reine Angew. Math. \tbf{310}, 56--80 (1979)


\end{thebibliography}
\end{document}